\documentclass[12pt,reqno,a4paper]{amsart}
\usepackage{amsmath,amssymb,amsfonts,amscd}
\usepackage[mathscr]{eucal}
\usepackage{comment}
\usepackage{hyperref}

\date{12  (25) September  2014; last modified 28 August (10 September) 2019}  

\author{Theodore~Th. Voronov}
\address{{Department of Mathematics,  University of Manchester,    Manchester,   M13 9PL,  UK}}
\email{theodore.voronov@manchester.ac.uk}
\address{%
{Faculty of Physics, Tomsk State University, Tomsk, 634050, Russia}}

\title[L-infinity bialgebroids and homotopy Poisson structures]{L-infinity bialgebroids and homotopy Poisson structures on supermanifolds}

\newtheorem{theorem}{Theorem}
\newtheorem{proposition}{Proposition}

\theoremstyle{definition}
\newtheorem{definition}{Definition}
\newtheorem{example}{Example}
\newtheorem{remark}{Remark}

\def\co{\colon\thinspace}

\renewcommand{\geq}{\geqslant}

\DeclareMathOperator{\ad}{ad}

\DeclareMathOperator{\Hom}{Hom}

\newcommand{\fun}{C^{\infty}}

\newcommand{\der}[2]{{\frac{\partial {#1}}{\partial {#2}}}}

\newcommand{\lder}[2]{{\partial {#1}/\partial {#2}}}

\newcommand{\R}[1]{{\mathbb R}^{#1}}

\newcommand{\Z}{{\mathbb Z_{2}}}
\newcommand{\ZZ}{{\mathbb Z}}

\newcommand{\p}{\partial}

\renewcommand{\a}{\alpha}

\newcommand{\e}{\varepsilon}
\newcommand{\s}{\sigma}
\newcommand{\f}{{\varphi}}
\renewcommand{\O}{\Omega}

\renewcommand{\o}{\omega}

\newcommand{\h}{\eta}

\newcommand{\la}{{\lambda}}
\newcommand{\x}{{\xi}}
\renewcommand{\d}{\delta}

\renewcommand{\L}{{\Lambda}}

\newcommand{\ft}{{\tilde f}}

\newcommand{\at}{{\tilde a}}
\newcommand{\bt}{{\tilde b}}
\newcommand{\ct}{{\tilde c}}
\newcommand{\ut}{{\tilde u}}

\newcommand{\Xt}{{\tilde X}}

\newcommand{\Ft}{{\tilde F}}

\newcommand{\Pt}{{\tilde P}}

\newcommand{\lat}{{\tilde \lambda}}

\DeclareMathOperator{\Vect}{\mathrm{Vect}}

\DeclareMathOperator{\Mult}{\mathfrak{A}}

\newcommand{\lsch}{{[\![}}
\newcommand{\rsch}{{]\!]}}

\newcommand{\Sinf}{S_{\infty}}
\newcommand{\Pinf}{P_{\infty}}
\newcommand{\Linf}{L_{\infty}}

\begin{document}

\begin{abstract}
We generalize to the homotopy case a result  of K.~Mackenzie and P.~Xu  on relation  between Lie bialgebroids and Poisson geometry. For a homotopy Poisson structure on a supermanifold $M$, we show that $(TM, T^*M)$ has a canonical structure of an $L_{\infty}$-bialgebroid. (Higher Koszul brackets on forms   introduced earlier by H.~Khudaverdian and the author  are part of one of its manifestations.) The underlying general construction is that of a ``(quasi)triangular'' $L_{\infty}$-bialgebroid,  which is a specialization of a ``(quasi)triangular'' homotopy Poisson structure.  We define both here.
\end{abstract}

\maketitle

\tableofcontents

\section{Introduction}


In this text, we outline a simple construction that generalizes ``quasi-triangular structure'' as in Drinfeld's quasi-triangular Lie bialgebras and its generalization to Lie bialgebroids due to K.~Mackenzie and P.~Xu~\cite{mackenzie:bialg}. (See Theorem~\ref{thm.main} in Section~\ref{sec.quasi}.) We came to it in relation with our notion of higher Koszul brackets~\cite{tv:higherpoisson}. The idea comes from the following observations.  For a Poisson manifold $M$, its cotangent bundle $T^*M$ becomes a Lie algebroid. This is classical, and a lot of Poisson geometry rests on this construction. (It is useful to consider various manifestations of  Lie algebroid structure, and in this case one of them is the Koszul bracket, i.e. an odd Poisson bracket induced in the algebra of differential forms on $M$.) One can observe that taken together with the natural Lie algebroid structure in the tangent bundle $TM$, this makes $(TM, T^*M)$ a Lie bialgebroid (and various related maps are in fact Lie bialgebroid morphisms). Mackenzie and Xu in~\cite{mackenzie:bialg} (see also~\cite{mackenzie:book2005}) noticed that for this ``cotangent Lie bialgebroid'' construction for a Poisson manifold there is an abstract form that also incorporates the classical notion due to Drinfeld of a (quasi)triangular Lie bialgebra. (As it is known, Drinfeld's quasi-triangular Lie bialgebras are distinguished by the fact that one of the Lie brackets is not just a cocycle with respect to the other, but a coboundary, --- hence their alternative name as coboundary Lie bialgebras.) In~\cite{tv:higherpoisson}, H.~Khudaverdian and the author considered the homotopy analog of the cotangent Lie algebroid construction. Namely, we considered homotopy Poisson manifolds\footnote{as well as \emph{homotopy symplectic manifolds}; a close notion became popular recently under the name ``shifted symplectic structure''} and for them defined an $\Linf$-algebroid structure in the cotangent bundle.\footnote{Paper~\cite{tv:higherpoisson} seems also to be one of the first instances when $\Linf$-algebroids were introduced.} (One manifestation of which is the higher Koszul brackets introduced in~\cite{tv:higherpoisson}, i.e. an $\Sinf$-algebra structure for differential forms on a homotopy Poisson manifold.) Thus for a homotopy Poisson manifold $M$, the pair $(TM, T^*M)$ has two $\Linf$-algebroid structures (one for $TM$ is the ordinary canonical Lie algebroid structure), which are compatible in some precise sense.
This makes it possible to describe this pair as an \emph{$\Linf$-bialgebroid}. Moreover, we observed the mechanism behind this $\Linf$-bialgebroid structure (induced by a homotopy Poisson structure), which can be formulated in a very simple geometric form (as a ``gauge transformation of a master Hamiltonian'') and which embraces the cases of classical Drinfeld's quasi-triangular Lie bialgebras and the   quasi-triangular Lie bialgebroids of Mackenzie--Xu. Exposition of this construction is the main purpose of this text.

We would like to make a remark on the notion of an $\Linf$-bialgebroid. In some cases, such as described above, it is self-evident when a structure is of that type. The most immediate ``homotopy'' generalization of a structure of a Lie bialgebroid  seems to be the case of a pair of vector bundles in duality $(E,E^*)$, both endowed with $\Linf$-structures, and so that a compatibility condition is satisfied. (The simplest way of formulating this condition is $(H_{E},H_{E^*})=0$, where $H_{E}$ and $H_{E^*}$ are the ``master Hamiltonians'' for the $\Linf$-structures in $E$ and $E^*$ respectively.) The case of the cotangent $\Linf$-bialgebroid of a homotopy Poisson manifold is of this type. One can argue however that this is not the most general definition possible; when I was speaking about the cotangent $\Linf$-bialgebroid construction at a conference on homological methods in summer 2014~\cite{tv:talk-july2014},  A.~Voronov    remarked that such   a structure (of a pair of compatible $\Linf$-algebroids) was possibly too restrictive as not allowing for   brackets of sections ``across'' $E$ and $E^*$, only within each vector bundle. A more general structure of a homotopy Lie bialgebroid was  later considered in~\cite{bashkirov-voronov:lbalg}.
Actually, it is   easy to provide for all possible brackets with multiple arguments taken from sections of $E$ and $E^*$ by using the language of  Hamiltonians. One just takes a general odd   ``master Hamiltonian''  $H$ on $T^*(\Pi E)\cong T^*(\Pi E^*)$ satisfying $(H,H)=0$, instead of a Hamiltonian of the form $H_E+H_{E^*}$ that would correspond to a pair of structures on $E$ and $E^*$. Algebraic structure of brackets that would result from that was studied in the PhD thesis of our student M.~Peddie~\cite{peddie:thesis} (see also~\cite{peddie:loday}). A remark to be made is that defining an $\Linf$-bialgebroid by a single geometric object such as a master Hamiltonian on the common cotangent or a BV type operator, in spite of the generality this suggests, can obscure the original idea that was underlying the introduction of Lie bialgebras in integrable systems theory, namely, of   \emph{two} compatible brackets structures that hence can be arranged in a pencil  (i.e. a one-parameter family). (It also blurs the difference between  a ``bi-'' object and its ``double''.)  Therefore, we think that the notion of an $\Linf$-bialgebroid should always include either a pair of compatible structures or other way of incorporating a pencil of structures. The ``quasi-triangular'' construction of homotopy Poisson brackets and $\Linf$-bialgebroids given here satisfies this requirement.

Most  of the text  was  written in August 2014 (apart from this introduction and the references to works after 2014). For various reasons, the main of which was the author's pre-occupation with developing the theory of \emph{thick morphisms}, the text was not published at the time. The   construction of ``(quasi)triangular'' homotopy Poisson structures and   ``(quasi)triangular'' $L_{\infty}$-bialgebroids  as a special case should be useful and have not been described anywhere except for brief references e.g. in~\cite{tv:nonlinearpullback} and \cite{tv:microformal}. Because of that, I decided  to make the text  available with minimal amendments (so in many places it remains  raw).

\emph{Thick} (or \emph{microformal}) \emph{morphisms} of supermanifolds were discovered shortly after the first version of this text was written. As it is clear to the author, any complete theory of homotopy Poisson structures and $\Linf$-(bi)algebroids has to include thick morphisms. As explained, they are not mentioned in what follows. Merging the subject of the present paper with microformal geometry  will be done elsewhere  as we hope. For references on thick morphisms see~\cite{tv:nonlinearpullback, tv:microformal, tv:gradedmicro}. See also~\cite{tv:highkosz}.

\subsection*{Notations}
When we speak about vector spaces and algebras we  assume that they are $\Z$-graded, e.g., $V=V_0\oplus V_1$\,. A $\Z$-grading (\emph{parity}), which is always assumed, should be distinguished from a $\ZZ$-grading, which may be present or not.
According to our  philosophy, a  $\ZZ$-grading  is a special feature which may arise from a linear structure such as that of a vector bundle and in general should be  regarded as its replacement.
When both kinds of gradings are present, they are, generally, independent. (On graded geometry see~\cite{tv:graded} and also~\cite{tv:gradedmicro}.)

Parity of an element is denoted by the tilde over the element's symbol.

\section{Preliminaries}

\subsection{Lie brackets and change of parity. Poisson and Schouten algebras}
For reference, we recall the definitions of a Lie superalgebra, Poisson algebra and Schouten (Gerstenhaber) algebra, with a particular emphasis on signs. In what follows, we use a generic notation with the square bracket as a Lie bracket and with the curly bracket as a Poisson or Schouten bracket. In particular examples, different notations for brackets may be used.

Recall that a  \emph{Lie superalgebra} $L=L_0\oplus L_1$ has a bracket with the following properties. It is even: $[L_i,L_j]\subset L_{i+j}$;   bilinear,   in particular,
\begin{equation}\label{eq.liebilin}
    [\la a,b]=\la [a,b],  \quad [a,b\la]= [a,b]\la\,,
\end{equation}
for a scalar $\la$ of arbitrary parity;  antisymmetric:
\begin{equation}\label{eq.lieantisym}
    [a,b]=-(-1)^{\at\bt}[b,a]\,;
\end{equation}
and satisfies the Jacobi identity:
\begin{equation}\label{eq.liejac}
    [a,[b,c]]=[[a,b],c]+(-1)^{\at\bt}[b,[a,c]]\,.
\end{equation}

Parity shift leads to a variant of a Lie superalgebra with an odd bracket. It exists in two versions. (Note that, in general, parity shift in combination with a   redefinition of the common sign transforms antisymmetric operations into symmetric and \emph{vice versa}.)

\emph{Version 1. } (Antisymmetric bracket.)
An \emph{odd Lie superalgebra}  $L$ has an odd bracket,  $[L_i,L_j]\subset L_{i+j+1}$, which is bilinear in the sense that
\begin{equation}\label{eq.lieoddbilin}
    [\la a,b]=\la[a,b]\,, \quad [a,b\la]=[a,b]\la\,,
\end{equation}
(in the same form as~\eqref{eq.liebilin})
and which satisfies the antisymmetry
\begin{equation}\label{eq.lieoddantisym}
    [a,b]=-(-1)^{(\at+1)(\bt+1)}[b,a]\,
\end{equation}
and   the Jacobi identity
\begin{equation}\label{eq.lieoddjac}
    [a,[b,c]]=[[a,b],c]+(-1)^{(\at+1)(\bt+1)}[b,[a,c]]\,
\end{equation}
w.r.t. the opposite parity.
One can see that here $[a\la,b]=(-1)^{\lat}[a,\la b]$, so the parity of the bracket  `sits' on the central comma.

\emph{Version 2. } (Symmetric bracket.)
An \emph{odd Lie superalgebra (symmetric version)} or a \emph{Lie antialgebra}  $L$ has an odd bracket,  $[L_i,L_j]\subset L_{i+j+1}$, satisfying   bilinearity in the form
\begin{equation}\label{eq.antiliebilin}
    [\la a,b]=(-1)^{\lat}\la[a,b]\,, \quad [a,b\la]=[a,b]\la\,,
\end{equation}
which is  symmetric:
\begin{equation}\label{eq.antiliesym}
    [a,b]=(-1)^{\at\bt}[b,a]\,,
\end{equation}
and  satisfies the Jacobi identity in the form
\begin{equation}\label{eq.antiliejac}
    [a,[b,c]]=(-1)^{\at+1}[[a,b],c]+(-1)^{(\at+1)(\bt+1)}[b,[a,c]]\,.
\end{equation}
The parity here  `sits'   on the opening bracket $[$\,, which should be regarded as an odd symbol. The form of the Jacobi identity for this version looks less attractive than that for the antisymmetric version, but in fact~\eqref{eq.antiliejac} can be equivalently re-written as
\begin{equation}
    [[a,b]c]+(-1)^{\bt\ct}[[a,c],b]+(-1)^{\at(\bt+\ct)+\bt\ct}[[c,b],a]=0\,,
\end{equation}
where we see the sum over   $(2,1)$-shuffles with the signs  due to the parities only (no input of permutation signs). In this form it  most conveniently extends to the `strongly homotopy' case (see below).

The two versions of odd Lie brackets can be transformed into each other by a re-definition of the common sign. If $[\_\,,\_]$ is an odd bracket on $L$ satisfying~\eqref{eq.lieoddbilin}, \eqref{eq.lieoddantisym}, \eqref{eq.lieoddjac}, then a  bracket $[\_\,,\_]'$ defined by
\begin{equation*}
    [a,b]':=(-1)^{\at}[a,b]
\end{equation*}
(or by $[a,b]':=(-1)^{\at+1}[a,b]$ \,---\,both options work)
satisfies~\eqref{eq.antiliebilin}, \eqref{eq.antiliesym} and \eqref{eq.antiliejac}; and the other way round.

If $L$ is a Lie superalgebra, with an even bracket satisfying the usual properties~\eqref{eq.liebilin}, \eqref{eq.lieantisym} and \eqref{eq.liejac}, then the formula
\begin{equation}\label{eq.brackoddanti}
    [\Pi a,\Pi b]:=\Pi [a,b]
\end{equation}
defines on $\Pi L$ an odd bracket satisfying~\eqref{eq.lieoddbilin}, \eqref{eq.lieoddantisym}, \eqref{eq.lieoddjac}. Alternatively,
the formula
\begin{equation}\label{eq.brackoddsym}
    [\Pi a,\Pi b]:=(-1)^{\at}\Pi [a,b]
\end{equation}
defines  on $\Pi L$ an odd bracket satisfying~\eqref{eq.antiliebilin}, \eqref{eq.antiliesym} and \eqref{eq.antiliejac}.

The adjoint representation is defined for Lie superalgebras in the usual way: $(\ad a)(b)=[a,b]$. The operator $\ad a$ has the same parity as $a$. In particular, for an odd $a$, if $[a,a]=0$, then $(\ad a)^2=\frac{1}{2}[\ad a, \ad a]=\frac{1}{2}\ad[a, a]=0$\,.

For odd versions of Lie superalgebras, the linear operator $b\mapsto [a,b]$ has the parity opposite to that of $a$. The adjoint representation is defined by $(\ad a)(b):=(-1)^{\at}[a,b]$ for the antisymmetric version (where the sign is needed for the linearity of $\ad$) and by $(\ad a)(b):=[a,b]$  for the symmetric version. If an even element $a$ satisfies $[a,a]=0$ for an odd bracket, then $(\ad a)^2=0$\,.

We call an odd element satisfying $[a,a]=0$ w.r.t. an even bracket or an even element satisfying $[a,a]=0$ w.r.t. an odd bracket, \emph{homological}. The identity $(\ad a)^2=0$ for a homological element is used many times in this paper.

Now let us recall Poisson and Schouten (Gerstenhaber) algebras. Both types of algebras have two multiplicative operations: a commutative associative product and a bracket. They differ by the relative parity of these operations (which cannot be changed by a parity shift). If we agree that the associative product is even, then the bracket can be even or can be odd, and there are two conventions for the latter.

A \emph{Poisson algebra} or   \emph{even Poisson algebra} $A$ is a commutative associative algebra,\footnote{The requirement of commutativity can be dropped.}   $A_iA_j\subset A_{i+j}$, endowed with an even   bracket, $\{A_i,A_j\}\subset A_{i+j}$, such that it forms  a Lie superalgebra w.r.t.  this bracket  and the Leibniz identity
\begin{equation}\label{eq.leibn}
    \{a,bc\}=\{a,b\}c+(-1)^{\at\bt}b\{a,c\}
\end{equation}
is satisfied.

A parallel definition of Schouten algebras goes in two versions.

\emph{Version 1. }
A \emph{Schouten algebra}  or  \emph{odd Poisson algebra} or  \emph{Gerstenhaber algebra} $A$ is a commutative associative algebra,  $A_iA_j\subset A_{i+j}$,  with an odd  bracket,  $\{A_i,A_j\}\subset A_{i+j+1}$, such that w.r.t.  the bracket it is  an odd Lie superalgebra  (meaning that ~\eqref{eq.lieoddbilin}, \eqref{eq.lieoddantisym}, \eqref{eq.lieoddjac} hold for the bracket) and the Leibniz identity (see below) is satisfied.

\emph{Version 2. }
A  \emph{Schouten algebra} or  \emph{odd Poisson algebra} or  \emph{Gerstenhaber algebra}   (\emph{in  symmetric convention}) is defined in the same way except for requiring that it forms a Lie antialgebra w.r.t. the odd bracket (that is, ~\eqref{eq.antiliebilin}, \eqref{eq.antiliesym} and \eqref{eq.antiliejac} hold for the bracket).

For both versions of the definition of a Schouten algebra, the Leibniz identity has the same form
\begin{equation}\label{eq.leibnodd}
    \{a,bc\}=\{a,b\}c+(-1)^{(\at+1)\bt}b\{a,c\}\,.
\end{equation}
(The Leibniz identities~\eqref{eq.leibn} and ~\eqref{eq.leibnodd} for even and odd brackets alike simply mean that taking a bracket is a derivation of the associative product.)

\subsection{Canonical brackets and the de Rham differential. Derived brackets}
Let $M$ be a supermanifold.
For further reference we give the explicit formulas for the canonical Poisson and Schouten brackets. On $T^*M$ we have the canonical even symplectic form   $\o=dp_adx^a=d(p_adx^a)=d(dx^ap_a)$. The corresponding  Poisson bracket is
\begin{equation*}
    (F,G)=(-1)^{\!\Ft\at}\left((-1)^{\at}\der{F}{p_a}\der{G}{x^a}- \der{F}{x^a}\der{G}{p_a}\,\right)\,,
\end{equation*}
for arbitrary $F,G\in \fun(T^*M)$\,. Together with the usual   multiplication  it  makes   $\fun(T^*M)$  a Poisson algebra so that  the `initial conditions'
\begin{align*}
    (f,g)&=0\,, \\
    (X\cdot p, f)&=\p_Xf\,,\\
    (X\cdot p, Y\cdot p)&=[X,Y]\cdot p\,,
\end{align*}
are satisfied, and this defines the bracket uniquely.
Here $f,g\in\fun(M)$, $X,Y\in\Vect(M)$, and $X\cdot p=X^ap_a$  denotes the fiberwise-linear Hamiltonian corresponding to a vector field $X=X^a\p_a$.

\begin{example} For an odd $F\in \fun(T^*M)$,
\begin{equation*}
    (F,F)=2\,\der{F}{p_a}\der{F}{x^a}\,.
\end{equation*}
We shall use this formula later.
\end{example}

Consider now the vector bundles $\Pi TM$ and $\Pi T^*M$.
They are supermanifolds even if $M$ is an ordinary manifold. $TM$, $T^*M$, $\Pi TM$ and $\Pi T^*M$ are the four  `neighbor'  bundles that can be obtained from each other by dualizations and parity shifts.
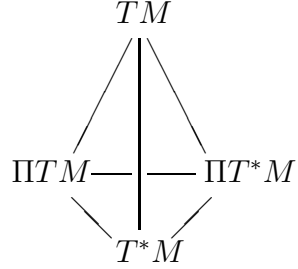
\begin{center}
{ \unitlength=3pt
\begin{picture}(0,35)
\put(0,30)
    {\begin{picture}(0,0)
    \put(-2,0){$TM$}
    \put(-15,-20){$\Pi TM$}
    \put(9,-20){$\Pi T^*M$}
    \put(-2,-30){$T^*M$}
    {\thicklines
    \put(1,-2){\line(0,-1){24}} }
    \put(-5,-19){\line(1,0){5}}
    \put(2,-19){\line(1,0){6}}
    \put(0,-2){\line(-1,-2){7.2}}
    \put(2,-2){\line(1,-2){7.2}}
    \put(-7.5,-22){\line(1,-1){5}}
    \put(10,-22){\line(-1,-1){5}}
    \end{picture}}
\end{picture}
}
\end{center}
Each bundle here
possesses a canonical structure, all these structures being essentially equivalent to each other and   just manifestations of one canonical `first order structure' for every supermanifold $M$, as is made explicit below. The difference between   $TM$ and the other natural bundles   $T^*M$, $\Pi TM$ and $\Pi T^*M$ is that the structure we are speaking about is defined on sections of $TM\to M$ (vector fields), while in the other cases the structure is defined on functions on the total space.

Functions on $\Pi TM$ can be identified with differential forms: $\fun(\Pi TM) = \O(M)$\,, and functions on $\Pi T^*M$ can be identified with multivector fields: $\fun(\Pi T^*M) = \Mult(M)$\,. (There is a certain abuse of language here, but we shall stick to it.) Local coordinates $x^a$ on $M$ induce coordinates on $\Pi TM$ and $\Pi T^*M$, which $x^a,dx^a$ and $x^a,x^*_a$, respectively. The parity of $dx^a$ and $x^*_a$ is opposite to that of $x^a$. The transformation laws are
 \begin{equation*}
    dx^a=dx^{a'}\der{x^a}{x^{a'}}
 \end{equation*}
and
\begin{equation*}
    x^*_a=\der{x^{a'}}{x^{a}}\, x^*_{a'}\,.
\end{equation*}
(So the `antimomenta'   $x^*_a$ transform exactly as the momenta $p_a$, but has the opposite parities. The same is true for $dx^a$ and the velocities $\dot x^a$.)

The canonical structure on $\Pi TM$ is of course the de Rham differential, which we regard as a vector field: $d\in \Vect(\Pi TM)$\,. In coordinates,
\begin{equation}\label{eq.d}
    d=dx^a \der{}{x^a}\,.
\end{equation}
The field $d$ is odd and since $d^2=0$ it is an example of a homological vector field. Relation between $d$ and the commutator of vector fields on $M$ itself is given by the Cartan formula
\begin{equation}\label{eq.commder}
    i_{[X,Y]}= (-1)^{\Xt}[[d,i_X],i_Y]\,.
\end{equation}
Here $i_X\in \Vect(\Pi TM)$ is the interior product with $X\in \Vect (M)$. Explicitly,
\begin{equation}\label{eq.inter}
    i_X=(-1)^{\Xt}X^a(x)\der{}{dx^a}\,,
\end{equation}
if $X=X^a\p_a$\,. (The signs in~\eqref{eq.commder} and~\eqref{eq.inter} are required for linearity.) Note that $[i_X,i_Y]=0$. This is a prototypal relation for derived brackets, see below.

\subsection{Homological vector fields. $L_{\infty}$-algebras}

\subsubsection{Homological vector fields and $Q$-manifolds}

A \emph{homological vector field} is a homological element of the Lie superalgebra of vector fields. In other words, an odd vector field $Q$ such that $Q^2=0$. A (super)manifold endowed with a homological vector field is called a \emph{$Q$-manifold}. A morphism of $Q$-manifolds or a $Q$-map ($Q$-morphism) is a smooth map that intertwines homological vector fields. (Remark. $Q$-manifolds, in particular in the presence of a $\ZZ$-grading and assuming additionally that $\deg Q=+1$, are sometimes referred to as ``DG-manifolds''.) The theory of $Q$-manifolds was initiated by A.~Schwarz (to whom belongs the name), A. Vaintrob and M.~Kontsevich.

\subsubsection{$L_{\infty}$-algebras (symmetric version)}
Consider  an odd vector field $Q\in \Vect (\R{m|n})$. Let its Taylor expansion at the origin have the form
\begin{equation*}
  Q=\left(Q_0^k+\xi^iQ^k_i+\frac{1}{2}\xi^j\xi^iQ_{ij}^k+
  \frac{1}{3!}\xi^l\xi^j\xi^iQ_{ijl}^k+\ldots\right)\der{}{\xi^k}.
\end{equation*}
The coefficients $Q_0^k$, $Q^k_i$, $Q_{ij}^k$, $Q_{ijl}^k$, \dots
define a sequence of $N$-ary operations ($N=0,1,2,3,\dots$) on the
vector space $\mathbb R^{n|m}=\Pi \R{m|n}$, and the condition $Q^2=0$
expands to a linked sequence of ``generalized Jacobi identities''. If only the
quadratic term   is
present, we return to the case of a Lie (super)algebra. The general case
is a  \emph{strongly homotopy Lie algebra}
(or \emph{$L_{\infty}$-algebra}). Usually $Q_0=0$ (otherwise the algebra is called ``curved'' or with ``background'').

{Coordinate-free description}
Given a superspace $V$. (For Lie algebras, $V=\mathfrak g$.) Each $v\in V$ defines a (constant) vector field $i_{v}\in V$. Define ``higher derived brackets'' as follows (here $N=0,1,2, \ldots, $):
\begin{equation*}
    i_{\{v_1,\ldots,v_N\}_Q}:=[[[\ldots[Q,v_1],v_2],\ldots,v_N](0)\,.
\end{equation*}

These operations    odd  and symmetric (in the super sense).
\begin{theorem}They satisfy the identities
\begin{equation*}
    \sum_{\parbox{1.2cm}{\small $\scriptstyle k+l=N$}} \sum_{\text{$(k,l)$-shuffles}}
    (-1)^{\a} \{\{v_{\s(1)},\ldots,v_{\s(k)}\},v_{\s({k+1})},\ldots,v_{\s({k+l})}\}=0
\end{equation*}
for all $N=0,1,2,\ldots\ $ if and only if $Q^2=0$. (Here $(-1)^{\a}$ is the sign
prescribed by the sign rule for a permutation of homogeneous
elements $v_1,\ldots, v_N\in V$.)
\end{theorem}

See more in~\cite{tv:higherder}.


A \emph{morphism} of $L_{\infty}$-algebras $V_1\to V_2$ (or an \emph{$L_{\infty}$-morphism}) is  a morphism of the corresponding $Q$-manifolds (i.e., a smooth map that relates $Q_1$ on $V_1$ and $Q_2$ on $V_2$).

In coordinates: if $Q_1=Q^k(\x)\der{}{\x^k}$ and $Q_2=Q^{\mu}(\h)\der{}{\h^{\mu}}$, one has to expand
\begin{equation*}
    Q_1^i(\x)\der{\h^{\mu}}{\x^i}=Q_2^{\mu}(\h(\x))
\end{equation*}
into a Taylor series at the origin. (Here $F\co (\x^i)\mapsto (\h^{\mu}(\x))$.)

{Example: `adjoint representation'.}
Chuang and Lazarev introduced an `adjoint representation' for $L_{\infty}$-algebras. In our language it looks as follows.

Consider an $L_{\infty}$-algebra $L$ in the symmetric version, where $Q\in \Vect(L)$ is the corresponding homological vector field. \\
The  \emph{adjoint representation} of $L$  is  a (non-linear) mapping $L\to \Pi\Vect(L)$ given by the formula:
\begin{equation*}
    \h \mapsto Q^{\h}- Q- Q(\h)\,,
\end{equation*}
where $Q^{\h}$ denotes a parallel shift, $Q^{\h}=Q^k(\x+\h)\der{}{\x^k}$ and $Q(\h)=Q^k(\h)\der{}{\x^k}$ is a constant vector. One can check that it is   an $L_{\infty}$-morphism.

\subsection{Poisson structures and homotopy Poisson structures}



For general construction of ``higher derived brackets'' see~\cite{tv:higherder, tv:higherderarb}.

A Poisson structure in the algebra $\fun(M)$ is specified by an even fiberwise quadratic function $P\in \fun(\Pi T^*M)$,
$P=\frac{1}{2}P^{ab}x^*_bx^*_a$, by the formula:
\begin{equation*}
    \{f,g\}_P=\lsch \lsch P,f\rsch,g\rsch\,.
\end{equation*}
In coordinates,
\begin{equation*}
    \{f,g\}_P=\pm P^{ab}\p_af\p_bg\,.
\end{equation*}
The condition $\lsch P, P \rsch =0$ is required and it gives the Jacobi identity for the bracket $ \{f,g\}_P$.
Similarly, an odd Poisson structure in the algebra $\fun(M)$ is specified by an odd fiberwise quadratic function $S\in \fun(T^*M)$ (a  `master Hamiltonian'), by the same kind of formulas. It should satisfy $(H,H)=0$. 

\begin{example} The master Hamiltonian $D=  \pi^a p_a\in \fun(T^*(\Pi T^*M))$ for the canonical Schouten bracket on $\fun(\Pi T^*M)$.
\end{example}

Similarly, a \emph{homotopy Poisson structure} or \emph{$\Pinf$-structure} is defined by an arbitrary even function $P\in \fun(\Pi T^*M)$ satisfying $\lsch P, P \rsch =0$ by the formulas for $k=0,1,2,\ldots$
\begin{equation*}
    \{f_1,\ldots,f_k\}_P:= \lsch \ldots \lsch\lsch P,f_1\rsch, f_2\rsch, \ldots , f_k\rsch_{|M}\,,
\end{equation*}
for $k=1,2,\ldots $\,, where $f_i\in\fun(M)$\,. The brackets of odd orders are odd, the brackets of  even orders are
even.
(Only the jet of $P$ at $M$ is needed.) The brackets are 
antisymmetric in the super sense and satisfying higher Jacobi identities. (They make a structure of an $\Linf$-algebra in the antisymmetric version.) They are also multiderivations with respect to associative multiplication.

Explicitly, if
\begin{equation*}
P=P(x,x^*)= P^a(x)\, x^*_a +\frac{1}{2}\, P^{ab}(x)\,
x^*_bx^*_a+ \frac{1}{3!}\,P^{abc}(x)\, x^*_cx^*_bx^*_a+\ldots\,,
\end{equation*}
then
\begin{equation*}
    \{f_1,\ldots,f_n\}_P =\pm
    P^{a_1\ldots a_k}(x)\,\p_{a_k}\!f_k\,\ldots\,
    \p_{a_1}\!f_1\,.
\end{equation*}

Similarly we have the notion of homotopy Schouten structure. 
A \emph{homotopy Schouten} or \emph{$\Sinf$ structure} is defined by an odd function (or formal Taylor series at $M$) $H\in\fun(T^*M)$ satisfying $(H,H)=0$, by
\begin{equation*}
    \{f_1,\ldots,f_k\}_H=  ( \ldots (H,f_1),\ldots,f_k)_{|M}\,.
\end{equation*}
All the brackets are odd, are derivations in each arguments and satisfy higher Jacobi identities (of an $\Linf$-algebra in symmetric form). Explicit formulas are written similarly.

Of course there are purely algebraic versions. 

A \emph{homotopy Poisson algebra} or a $\Pinf$-algebra is a commutative associative algebra which is also an $L_{\infty}$-algebra {\small (in the antisymmetric version)} such that all  brackets are multiderivations of the associative product.

A \emph{homotopy Schouten algebra} or $\Sinf$-algebra is a commutative associative algebra which is also an $L_{\infty}$-antialgebra {\small (or an `$L_{\infty}$-algebra in the symmetric version')} such that all  brackets are multiderivations of the associative product.

\subsection{The cotangent double vector bundle  and its odd analog} \label{subsec.cot}

We shall use the following commutative diagram and some standard notation associated with it:
\begin{equation*}
    \begin{CD} T^*E @>>> E^* \\
                @VVV    @VVV\\
                E @>>> M
    \end{CD}
\end{equation*}
Here $E$ is an arbitrary vector bundle. The horizontal projection $T^*E\to E^*$ arises from the natural identification $T^*E\cong T^*E^*$ discovered by Mackenzie  and Xu.\footnote{Their theorem has predecessors: namely, Tulczyjew, for $E=TM$, and Dufour, who discovered the general case  about the same time as Mackenzie and Xu, but whose work was unpublished.} Such an identification is unique up to some sign, and a choice of sign may give either a symplectomorphism or antisymplectomorphism.  A particular choice is not important here. What is important, is that this diagram introduces into $T^*E$ a double vector bundle structure and the associated bi-grading. There are the following \emph{weights} on $T^*E$: $w_1$ is the standard `vertical' weight (degree in the momentum variables for $T^*E\to E$) and $w_2$ is the standard `horizontal' weight (which corresponds to degree in the momentum variables for $T^*E^*\to E^*$).

There is an odd analog of the Mackenzie--Xu theorem (see~\cite{tv:graded}). It gives a natural identification $\Pi T^*E\cong \Pi T^*(\Pi E^*)$. (Again, there is a   choice of signs.) From that, we have a commutative diagram
\begin{equation*}
    \begin{CD} \Pi T^*E @>>> \Pi E^* \\
                @VVV    @VVV\\
                E @>>> M
    \end{CD}
\end{equation*}
and a bi-graded structure of $\Pi T^*E$.

\subsection{Lie algebroids and $L_{\infty}$-algebroids}

We shall consider Lie algebroids via their cotangent bundles. (Such an approach is the `cotangent philosophy' championed by Kirill Mackenzie.)


A \emph{Lie algebroid} over $M$ is a vector bundle $E\to M$ with a Lie algebra structure on the space of sections $\fun(M,E)$ and a bundle map $a\co E\to TM$ (called the \emph{anchor}) satisfying
\begin{equation*}
    [u,fv]={a(u)}f\,v+(-1)^{\ut\ft}f[u,v]
\end{equation*}
($u\in \fun(M,E)$ and $f\in \fun(M)$).

Examples: a Lie (super)algebra $\mathfrak g$ (here $M=\{*\}$); the tangent bundle $TM\to M$; an integrable distribution $D\subset TM$; an ``action algebroid'' $M\times \mathfrak g$.

Equivalent manifestations on ``neighbors'':
\begin{itemize}
\item Homological vector field of weight $1$ on $\Pi E$
\item Poisson bracket of weight $-1$ on $E^*$
\item Schouten bracket of weight $-1$ on $\Pi E^*$
\end{itemize}
(structures on total spaces!). {Description via $Q$-manifolds.} 
In local coordinates $x^a$ (on the base) and $\x^i$ (on the fibers), we have on $\Pi E$:
\begin{equation*}
    Q=\x^iQ_i^a(x)\,\der{}{x^a}+\frac{1}{2}\,\x^i\x^j Q_{ji}^k
    (x)\,\der{}{\x^k}\,.
\end{equation*}
The anchor and the Lie bracket for $E$ are expressed by
\begin{equation*}
    a(u)f:=\bigl[[Q,i_u)],f\bigr]
\end{equation*}
and
\begin{equation*}
    i_{[u,v]}):=(-1)^{\ut}\bigl[[Q,i_u],i_v\bigr].
\end{equation*}
Here the map $i\co \fun(M, E)\to \Vect (\Pi E)$  is
$i_u =(-1)^{\ut}u^i(x)\der{}{\x^i}$.

 {Morphisms of Lie algebroids.}
The definition of a \emph{morphism} of Lie algebroids over different bases (due to Higgins and Mackenzie) is tricky. It is a morphism of vector bundles
\begin{equation*}
    \begin{CD} E_1 @>{\Phi}>> E_2 \\
                @VVV    @VVV\\
                M_1 @>{\varphi}>> M_2
    \end{CD}
\end{equation*}
satisfying non-obvious conditions.

\begin{proposition}[Vaintrob] This vector bundle map
is a morphism of Lie algebroids if and only if
the  induced map $\Phi^{\Pi}\co \Pi E_1\to \Pi E_2$ of the opposite vector bundles
is a morphism of $Q$-manifolds.
\end{proposition}



We need to recall Lie bialgebroids.

We use the following language: a \emph{$P$-manifold} is a Poisson manifold; an \emph{$S$-manifold} is a Schouten manifold; a \emph{$QP$-manifold} (a \emph{$QS$-manifold})  possesses both $Q$- and $P$-structure ($S$-structure, resp.) so that the vector field is a derivation of the bracket. See more in~\cite{tv:graded}.

{Lie bialgebroids} were introduced by Mackenzie and Xu; an efficient description later found by Y.~Kosmann-Schwarzbach. Below is a version that uses the language of $Q$-manifolds.

A \emph{Lie bialgebroid} over $M$ is a Lie algebroid $E$ over $M$ such that $E^*$ is also a Lie algebroid over $M$ and so that   $\Pi E$ (with the induced structure) is a $QS$-manifold. Equivalently: $\Pi E^*$ is a $QS$-manifold.

\section{The $L_{\infty}$-bialgebroid associated with a homotopy Poisson structure}

\subsection{Classical construction of the Koszul bracket}\label{subsec.koszul}

The name `Koszul bracket' has several meanings. Here by a Koszul bracket we understand a Lie bracket on $1$-forms defined by a Poisson structure and its extension to a Schouten algebra structure on arbitrary differential forms (and in the super case\ --- \ on pseudodifferential forms).

Let $M$ be a Poisson manifold  with a Poisson tensor $P=(P^{ab}(x))$. This of course gives the Poisson  bracket of functions:
\begin{equation*}
    \{f,g\}_P=P^{ab}\p_af\p_bg\,.
\end{equation*}
In particular, we have the brackets of coordinates
\begin{equation*}
    \{x^a,x^b\}_P=P^{ab}(x)\,.
\end{equation*}

Besides that, a Poisson structure also gives:

\begin{itemize}
  \item The Lichnerowicz differential $d_P$ on multivector fields;
  \item The Koszul bracket $[\_,\_]_P$ on differential forms.
\end{itemize}

The \emph{Lichnerowicz differential} $d_P$ corresponding to a Poisson structure is defined as follows:
\begin{itemize}
  \item  $d_P\co \Mult^k(M)\to \Mult^{k+1}(M)$  where $\Mult^k(M)$ is the space of multivector fields on
$M$ of degree $k$,
\begin{equation*}
    d_P:=\lsch P, \_\rsch\,.
\end{equation*}
Here $P=\frac{1}{2}P^{ab}x^*_bx^*_a$ and $\lsch \_, \_\rsch$ denotes the canonical Schouten bracket of multivector fields.
  \item Explicitly:
  \begin{equation*}
    d_P=P^{ab}x^*_b\der{}{x^a}+\frac{1}{2}\p_aP^{bc}x^*_cx^*_b\der{}{x^*_a}
  \end{equation*}
\end{itemize}

The \emph{Koszul bracket} on the algebra of forms $\O(M)$ induced by a Poisson structure on $M$ is defined by the properties:
\begin{itemize}
  \item it is an odd Poisson   bracket on  $\O(M)$,
  \item the `initial conditions' are satisfied:
  \begin{align*}
    [f,g]_P&=0\,, \\
    [f,dg]_P&=(-1)^{\ft}\{f,g\}_P\,, \\
    \quad [df,dg]_P&=-(-1)^{\ft}d\{f,g\}_P\,,
  \end{align*}
\end{itemize}
(signs for the super case). Here $\{\_,\_\}_P$ is the Poisson bracket of functions, $[\_,\_]_P$ is the Koszul bracket of forms.

In particular, for coordinates and their differentials we
have
\begin{equation}\label{koszul.eq.binary}
    [x^a,x^b]_P=0\,, \quad [x^a,dx^b]_P=-P^{ab}\,, \quad \text{and} \quad
    [dx^a,dx^b]_P= dP^{ab}\,.
\end{equation}

Denote by $\f_P^*\co \O(M)\to \Mult(M)$ the operation of `raising indices' with the help of
$P^{ab}$. It is classically known that:
\begin{itemize}
  \item there is a commutative  diagram
\begin{equation*}\label{intro.eq.diagram}
    \begin{CD} \Mult^k (M)@>{d_P}>> \Mult^{k+1}(M)\\
                @A{\f_P^*}AA         @AA{\f_P^*}A\\
                \O^k(M)@>{d}>> \O^{k+1}(M) \,;
    \end{CD}
\end{equation*}
  \item ${\f_P^*}$  maps the Koszul bracket $[\_,\_]_P$ to the canonical Schouten bracket $\lsch \_, \_\rsch$\,.
\end{itemize}

 {From the viewpoint of Lie (bi)algebroids.}
\begin{itemize}
  \item The Lichnerowicz differential $d_P$ and the Koszul bracket $[\_,\_]_P$   {\small (equivalently)}  specify  a Lie algebroid structure in  $T^*M$.\\
{\small Here $d_P\in \Vect(\fun(\Pi T^*M))$ is the associated homological vector field and $[\_,\_]_P$ is the associated Lie--Schouten bracket. }
  \item Compare with the canonical Lie algebroid structure of $TM$ \\{\small(for which the associated objects are $d$ and $\lsch\_,\_\rsch$). }
  \item It is easy to see that
\begin{itemize}
  \item $d$ is a derivation of the Koszul bracket $[\_,\_]_P$\,;
  \item $d_P$ is a derivation of the canonical Schouten bracket $\lsch\_,\_\rsch$\,.
\end{itemize}
\end{itemize}

That means that  $(TM, T^*M)$ is a Lie bialgebroid, called the  {canonical Lie bialgebroid} of a Poisson manifold (Mackenzie--Xu).

\subsection{Higher Koszul brackets and an $L_{\infty}$-bialgebroid structure}
\label{subsec.highkoszul}

Consider $M$ with a higher Poisson structure $P$. We wish to construct odd brackets on on $\fun(\Pi TM)$ that will be an analog of the Koszul bracket for an ordinary Poisson structure.
\begin{itemize}
  \item Given: $P\in \fun(\Pi T^*M)$,  $\tilde P=0$, and $\lsch P,P\rsch=0$.
  \item Required: $K=K_P\in \fun(T^*(\Pi T^*M))$, $\tilde K=1$, and $(K,K)=0$.
\end{itemize}
(It is also required that when $P$ is fiberwise quadratic, $K_P$ generates the classical Koszul bracket.)

\begin{theorem} There is  a natural odd linear map
\begin{equation*}
   \alpha \co \fun(\Pi T^*M)\to \fun(T^*(\Pi TM))
\end{equation*}
that takes the canonical Schouten bracket on $\Pi T^*M$
to the canonical Poisson bracket on $T^*(\Pi TM)$  up to a sign:
\begin{equation*}
   \a\left(\lsch P,Q\rsch\right)= (-1)^{\Pt+1}\bigl(\a(P), \a(Q)\bigr)\,,
\end{equation*}
for arbitrary $P,Q\in \fun(\Pi T^*M)$.
\end{theorem}
Therefore, for our purposes we can set $K_P:=\a(P)$. Since $\lsch P,P\rsch=0$, we have $(K_P,K_P)=0$ as required.

A proof of the theorem is based on the diagram:
\begin{equation*}
    \begin{CD} T^*(\Pi TM)=T^*(\Pi T^*M)@>>>   \Pi T^*M\\
                @VVV  @VVV \\
               \Pi TM@>>>M
    \end{CD}
\end{equation*}
The map $\a$ is the composition of the following maps preserving the brackets:
\begin{equation*}
    \fun(\Pi T^*M)\to \Vect(\Pi T^*M) \to \fun(T^*(\Pi T^*M))\to \fun(T^*(\Pi TM))\,.
\end{equation*}
From the representation $\lsch P,\_\rsch =((D,P),\_)$ we obtain also that
$\a(P)=(D,P)$. Here $D=dx^ap_a=\pm \pi^ap_a$ is the canonical odd Hamiltonian generating the Schouten bracket.

 {Explicit formulas for $K_P$ go as follows.} 
If $P=P(x,x^*)$,
then for $K_P\in \fun(T^*(\Pi TM))$\,,
\begin{equation*}
K_P=(-1)^{\at}\der{P}{x^*_a}(x,\pi_{.})p_a+dx^a\,\der{P}{x^a}(x,\pi_{.})\,,
\end{equation*}
where $\pi_{.}=(\pi_a)$ and we denote by $p_a,\pi_a$ the momenta conjugate to the coordinates $x^a,dx^a$ on $\Pi TM$, respectively. Note the linear dependence on the coordinates $dx^a$.

 {Explicit formulas for higher Koszul brackets will be as follows.}
The odd Hamiltonian $K=K_P$ defines a sequence of odd
$n$-ary brackets on $\O(M)=\fun(\Pi TM)$,
\begin{equation}\label{koszul.eq.higher}
    [\o_1,\ldots,\o_n]_P=(\ldots(K,\o_1),\ldots,\o_n)|_{\Pi TM}\,, \quad  n=0,1,2,\ldots \ ,
\end{equation}
which makes it a particular case of a
homotopy Schouten algebra.  It is instructive to have a look at the $(k+\ell)$-bracket of $k$ functions $f_1,\ldots,f_k$ and $\ell$ differentials $df_{k+1},\ldots, df_{k+\ell}$:
\begin{align*}
    [f_1,\ldots,f_k, df_{k+1},\ldots,df_{k+\ell}]&=0 \ \text{for $k\geq 2$}\,,  \\
    [f_1,df_2,\ldots,df_\ell]_P&=(-1)^{\e}\,\{f_1,f_2,\ldots,f_\ell\}_P\,,   \\
    [df_1,\ldots,df_\ell]_P&=(-1)^{\e+1}\,d\{f_1,\ldots,f_\ell\}_P\,,
\end{align*}
where $\e=(\ell-1)\ft_1+(\ell-2)\ft_2+\ldots+\ft_{\ell-1}+\ell$.

\begin{theorem}
\begin{itemize}
  \item The higher Koszul brackets define on $T^*M$ a structure of an $L_{\infty}$-algebroid;
  \item Together with the canonical algebroid structure of $TM$, this defines an $L_{\infty}$-bialgebroid $(TM,T^*M)$ canonically associated with a higher Poisson structure on $M$.
\end{itemize}
\end{theorem}
(This is a generalization of the theorem of Mackenzie--Xu.)

\subsubsection{$L_{\infty}$-algebroids and $L_{\infty}$-bialgebroids}

\begin{definition}
An \emph{$L_{\infty}$-algebroid} is a vector bundle $E\to M$ endowed with antisymmetric brackets on sections, of alternating parities, making the space $\fun(M,E)$ an $L_{\infty}$-algebra and with antisymmetric multilinear maps $E\times_M\ldots \times_ME\to TM$ (`higher anchors'), also of alternating parities, so that
\begin{equation*}
    [u_1,\ldots,u_k,fu_{k+1}]= a(u_1,\ldots,u_k)(f)\,u_{k+1}+(-1)^{\e} f[u_1,\ldots,u_k,u_{k+1}]
\end{equation*}
(where $(-1)^{\e}$ is given by the sign rule).
\end{definition}

Similarly one defines an  \emph{$L_{\infty}$-algebroid (symmetric version)}, where all brackets are odd, or an  \emph{$L_{\infty}$-antialgebroid}. If $E$ is an $L_{\infty}$-algebroid, then $\Pi E$ is $L_{\infty}$-antialgebroid. The structure is specified by a homological vector field $Q\in \Vect(\Pi E)$ such that $Q_{|M}=0$. Alternatively, it is specified by a sequence of Lie--Schouten brackets on $\fun(\Pi E^*)$.

(An $L_{\infty}$-algebroid in $E$ has ``manifestations'' on $\Pi E$, $E^*$ and $\Pi E^*$.)

Now we get to $L_{\infty}$-bialgebroids. One version of the notion (which is not the most general) is the following.  
\begin{definition}
An \emph{$L_{\infty}$-bialgebroid} is a vector bundle $E\to M$ such that both $E$ itself and its dual $E^*$ are $L_{\infty}$-algebroids, and these structures are compatible in the following way: the Hamiltonian functions corresponding to the homological vector fields $Q_{E}$ and $Q_{E^*}$ commute under the canonical Poisson bracket on $T^*(\Pi E)=T^*(\Pi E^*)$.
\end{definition}
{\small
This is equivalent to saying that the double vector bundle
\begin{equation*}
    \begin{CD} T^*(\Pi E)=T^*(\Pi E^*)@>>>   \Pi E^*\\
                @VVV  @VVV \\
               \Pi E@>>>M
    \end{CD}
\end{equation*}
is a `double $L_{\infty}$-algebroid'.
}

Commuting of the odd Hamiltonians $H_E$ and $H_{E^*}$ implies that their sum $H=H_E+H_{E^*}$ is self-commuting and defines  $\Sinf$-structures on $\Pi E$ and $\Pi E^*$ (note that a $Q$-structure is a special case of $\Sinf$). 

A natural generalization would be to consider an arbitrary odd function $H$ on $T^*(\Pi E)\cong T^*(\Pi E^*)$ satisfying $(H,H)=0$ and call that an  \emph{$L_{\infty}$-bialgebroid} structure (either on $E$ or $E^*$). In order to obtain a structure on sections, we   expand the master Hamiltonian $H$ according to the two weights $w_1$ and $w_2$ and consider the higher derived brackets restricted to Hamiltonians corresponding to sections of $E$ and $E^*$ (which can be specified by weights, as functions of weights $(1,0)$ and $(0,1)$).

(Note however that in such as way we lose the difference between a ``bialgebroid'' and its ``Drinfeld double''.)

\section{Quasi-triangular $S_{\infty}$- and $P_{\infty}$-structures. Quasi-triangular $L_{\infty}$-bialgebroids}
\label{sec.quasi}

\subsection{Triangular (and quasi-triangular) Lie bialgebroids of Mackenzie--Xu}

Mac\-kenzie and Xu gave a construction of   Lie bialgebroids generalizing Drinfeld's (quasi)tri\-an\-gular Lie bialgebras.  In our language it goes as follows.

Suppose $E\to M$ is a Lie algebroid. Consider the dual vector bundle $E^*\to M$ and choose a function $r\in \fun(\Pi E^*)$, which should be   even fiberwise-quadratic and  satisfy the master equation
\begin{equation*}
    \{r,r\}_E=0\,,
\end{equation*}
where $\{\_\,,\_\}_E$ is the Lie--Schouten bracket induced on $\fun(\Pi E^*)$. Such a datum makes it possible to endow the bundle $E^*\to M$ with a Lie algebroid structure so that together with the original structure on $E$ the pair $(E,E^*)$ becomes a Lie bialgebroid.

Namely, consider the cotangent bundle $T^*(\Pi E)$. We consistently use the identification $T^*(\Pi E)= T^*(\Pi E^*)$ given by the Mackenzie--Xu theorem and use weights $w_1$ and $w_2$ arising from the double vector bundle structure (see subsection~\ref{subsec.cot}).
Let $H_E\in\fun(T^*(\Pi E))$ denote the odd master Hamiltonian corresponding to the Lie algebroid structure of $E$. It corresponds, on one hand, to the homological vector field on $\Pi E$, and on the other hand, generates the Lie--Schouten bracket on $\Pi E^*$. Therefore it has weights $w_1(H_E)=1$ and $w_2(H_E)=2$. Define the Hamiltonian $H_{E^*}:= (H_E,r)$. It is odd, and since $w_1(r)=2$ and $w_2(r)=0$, we obtain  $w_1(H_{E^*})= 1+2-1=2$ and $w_2(H_{E^*})= 2+0-1=1$. Therefore  $H_{E^*}$ corresponds to an odd vector field on $\Pi E^*$ of weight $1$.
We have $(H_{E^*},H_{E^*})=((H_E,r),(H_E,r))=(((H_E,r),H_E),r)- (H_E,((H_E,r),r)))=  - (H_E,((H_E,r),r)))$. The latter is zero because by the assumption  $\{r,r\}_E=((H_E,r),r)=0$. Hence we have a  homological vector field and it makes $E^*\to M$ a Lie algebroid. Since by the construction $(H_E,H_{E^*})=(H_E,(H_E,r))=0$, the compatibility condition for the Lie algebroid structures on $E$ and $E^*$ is satisfied automatically, so the pair $(E,E^*)$ is a Lie bialgebroid.

A Lie bialgebroid $(E,E^*)$ constructed this way  is  called  \emph{triangular}.

\begin{example}[triangular Lie bialgebras]
Let $\mathfrak{g}$ be a Lie (super)algebra and $r\in \L^2(\mathfrak{g})$. Upon the identification $\L^2(\mathfrak{g})\cong S^2(\Pi \mathfrak{g})$, an element $r$ can be regarded as a quadratic function on $\Pi \mathfrak{g}^*$. Suppose $r$ is even (for ordinary Lie algebras this is automatic) and $[r,r]=0$, where the bracket here is the Lie--Schouten bracket. The above construction is equivalent in the conventional language to taking the coboundary $dr$ of $r$, where $dr\in C^1(\mathfrak{g},\L^2(\mathfrak{g}))=\Hom(\mathfrak{g},\L^2(\mathfrak{g}))$. The linear map $\d:=dr\co \mathfrak{g}\to \L^2(\mathfrak{g})$ can be regarded as a `cobracket' on the vector space $\mathfrak{g}$ or, dually, a bracket on $\mathfrak{g}^*$. The Jacobi identity and the compatibility with the bracket in $\mathfrak{g}$ follow as above. Thus the pair $(\mathfrak{g}, \mathfrak{g}^*)$ becomes a \emph{Lie bialgebra} (or its super version) in the sense of Drinfeld.
\end{example}

\begin{remark}[`triangular' vs `quasi-triangular'] Let us clarify the terminology. The adjective `triangular'   originates from the `triangle equation', an alternative name for the quantum Yang--Baxter equation for the ``quantum  $R$-matrix'', for which the equation $[r,r]=0$ in a Lie algebra was introduced as a classical analog (the ``classical Yang--Baxter equation''). Lie bialgebras arose as infinitesimal objects for Poisson--Lie (in the original terminology\,---\,Hamilton--Lie) groups introduced by Reiman and Semenov-Tianshansky. When a Lie algebra $\mathfrak{g}$ is simple, the compatibility condition that the cobracket is a $\mathfrak{g}$-cocycle implies that it is a coboundary, i.e., comes from some element $r\in \L^2(\mathfrak{g})$; however, $r$ does \emph{not} have to satisfy  the Yang--Baxter equation $[r,r]=0$ for the Jacobi identity for the cobracket to hold, only the weaker condition that the element $[r,r]$ is $\mathfrak{g}$-invariant. Lie bialgebras defined by   $r$ such that $[r,r]=0$ are called \emph{triangular}; in general, with $[r,r]$   invariant, they are called \emph{quasi-triangular}.
The latter are   main examples of Lie bialgebras.
\end{remark}

The Mackenzie--Xu construction is easily modified to include the case when  the function $r$ in the above setup  satisfies only
a ``generalized Yang--Baxter equation''
\begin{equation*}
   (H_E, \{r,r\}_E)=0
\end{equation*}
instead of the stronger requirement $\{r,r\}_E$=0. We shall call a Lie bialgebroid $(E,E^*)$ obtained this way,   \emph{quasi-triangular}  or \emph{coboundary}. They are a direct generalization of Drinfeld's quasi-triangular Lie bialgebras.

The Mackenzie--Xu construction can be also thought of as an abstract form of the following example (where $E=TM$).

\begin{example}[Lie bialgebroid associated with a Poisson manifold] This is the classical construction of the Koszul bracket (see the previous section). If $M$ is a Poisson (super)manifold with a Poisson tensor $P$, the latter is a fiberwise-quadratic function on $\Pi T^*M$ satisfying
\begin{equation*}
    \lsch P,P\rsch =0\,,
\end{equation*}
where $\lsch \_,\_\rsch$ stands for the canonical Schouten bracket.
Let $D\in \fun(T^*(\Pi TM))$ be the canonical Hamiltonian corresponding   to the de Rham differential on $\Pi TM$ and the Schouten bracket on $\Pi T^*M$. Since $\lsch P,P\rsch=((D,P),P)$, we are in the setup of the Mackenzie--Xu construction applied to the tangent Lie algebroid $TM\to M$. The new Hamiltonian $K:=(D,P)$ defines the Koszul bracket on $\Pi TM$. The pair $(TM,T^*M)$ is a triangular Lie bialgebroid called the \emph{canonical Lie bialgebroid} of a Poisson manifold. See details  in subsection~\ref{subsec.koszul} and    Mackenzie book for the standard exposition.
\end{example}


\subsection{Quasi-triangular $S_{\infty}$- and $P_{\infty}$-structures.}

Our aim will be to show that, exactly as the classical Koszul bracket fits into the Mackenzie--Xu construction of a triangular Lie bialgebroid, our construction of the  $L_{\infty}$-bialgebroid associated with a homotopy Poisson structure as described in subsection~\ref{subsec.highkoszul} generalizes to the notion of a  (quasi)triangular $L_{\infty}$-bialgebroid  (that we shall introduce). To this end, in this subsection,  we start with an abstract construction applicable to   arbitrary homotopy Schouten or homotopy Poisson manifolds.

Let $M$ be a supermanifold.
Consider an odd Hamiltonian $H\in \fun(T^*M)$ satisfying
\begin{equation}\label{eq.masterh}
    (H,H)=0\,,
\end{equation}
for the canonical Poisson bracket on $T^*M$. That means that $H$  specifies a homotopy Schouten (or $S_{\infty}$-) structure on $M$.

Let $r\in \fun(M)$ be a completely arbitrary even function. Let us stress that we do not assume at the moment  that the function $r$ satisfy any equations whatsoever.
Define a new odd Hamiltonian $H'$ by the formula:
\begin{equation}
    H'(x,p):= H\!\Bigl(x,p+\der{r}{x}\Bigr).\label{eq.newh}
\end{equation}
Since $r$ is even, the partial derivatives $\lder{r}{x^a}$ have the correct parities and the substitution $p_a+\lder{r}{x^a}$ for $p_a$ in the argument of $H$ is well-defined.

\begin{theorem} \label{thm.main}
The following holds:
\begin{enumerate}
  \item The function $H'\in \fun(T^*M)$ satisfies
\begin{equation}
    (H',H')=0\,.
\end{equation}
  \item If $r\in \fun(M)$ satisfies the `master equation'
  \begin{equation} \label{eq.masterhamjac}
    H\!\Bigl(x,\der{r}{x}\Bigr)=0,
  \end{equation}
  then the Hamiltonian $H'$ vanishes at the zero section $M\subset T^*M$.
  \item
  For a fiberwise-quadratic $H$, the above master equation~\eqref{eq.masterhamjac} takes the form
  \begin{equation}
    \{r,r\}_H=0\,,
  \end{equation}
  where the bracket is the Schouten bracket defined by $H$; if it is satisfied by the function $r$, then $H'$ takes the `coboundary' form
  \begin{equation}
    H'=H+(H,r)\,.
  \end{equation}
\end{enumerate}
\end{theorem}
\begin{proof} The shift of   argument of $H(x,p)$ in   formula~\eqref{eq.newh} defining $H'$ is nothing but a particular case of the canonical transformation given by a Hamiltonian flow on $T^*M$ (with the Hamiltonian function $-r$). Hence   the equality $(H',H')=0$ follows from the fact that Hamiltonian flows preserve the canonical Poisson bracket.
This proves part 1. Note that we nowhere used any conditions for the function $r$. Part 2 is obvious by the definition of $H'$. Finally, to prove part 3, consider the case of a quadratic $H$,
\begin{equation*}
    H=\frac{1}{2}\,H^{ab}(x)p_bp_a\,.
\end{equation*}
Then
\begin{equation*}
    H\!\Bigl(x,\der{r}{x}\Bigr)=\frac{1}{2}\,H^{ab}(x)\der{r}{x^b}\der{r}{x^a}=\frac{1}{2}\{r,r\}_H\,,
\end{equation*}
and
\begin{multline*}
    H'(x,p):= H\!\Bigl(x,p+\der{r}{x}\Bigr)=\frac{1}{2}\,H^{ab}(x)\left(p_b+\der{r}{x^b}\right)\left(p_a+\der{r}{x^a}\right)=\\
    \frac{1}{2}\,H^{ab}(x)p_bp_a +
     H^{ab}(x) p_b \der{r}{x^a} +
    \frac{1}{2}\,H^{ab}(x) \der{r}{x^b} \der{r}{x^a}\,.
\end{multline*}
In other words,  for   a quadratic $H$,
\begin{equation*}
    H'=H+(H,r)+\frac{1}{2}\{r,r\}_H\,,
\end{equation*}
and part 3 follows.
\end{proof}

We shall call the new $S_{\infty}$-structure   $H'$ given by~\eqref{eq.newh},  a \emph{quasi-triangular} or \emph{coboundary  $S_{\infty}$-structure} generated by an $S_{\infty}$-structure   $H$ and a \emph{``classical $r$-matrix''} $r$.

The meaning of the additional conditions that may be imposed on $r$ is as follows.

If $r$ satisfies
\begin{equation*}
   Q_H[f]\equiv H\!\Bigl(x,\der{f}{x}\Bigr)=0
\end{equation*}
then $H'(x,0)=0$, hence there is no $0$-bracket in the $S_{\infty}$-structure generated by $H'$. This is a sufficient, but not necessary. A weaker sufficient condition is
\begin{equation*}
    \left(H, Q_H[f]\right)=0\,.
\end{equation*}

Remark: strictly speaking, this is an analog of the Drinfeld double of a quasi-triangular Lie bialgebra (not the Lie bialgebra as such). Another remark: a parameter $t$ can be included, as a factor before $\der{r}{x}$.


Note that without the condition of vanishing of
\begin{equation*}
   Q_H[f]\equiv H\!\Bigl(x,\der{f}{x}\Bigr)
\end{equation*}
(for a quadratic $H$, of vanishing
of
$\{f,f\}_H$), the new odd Hamiltonian $H'$ has a non-zero restriction on $M$. Hence, the $L_{\infty}$-algebra for the new homotopy Schouten structure defined by $H'$ will be \emph{a priori} `with background' (or `curved'). If, however, the equation
\begin{equation*}
    \left(H, Q_H[f]\right)=0
\end{equation*}
is satisfied, then ``the curvature is abelian'', i.e., annihilates all brackets derived from $H$\,---\,in other words, the  $L_{\infty}$ structure will be strict. For quadratic case, the new master  Hamiltonian $H'$ contains only three terms in the momentum expansion, which corresponds to a differential Schouten structure `with background'. There is no background (algebraically) if $(H,\{f,f\}_H)=0$. In the example of Drinfeld's Lie bialgebras, this corresponds to the quasi-triangular case. So there is a terminological question, what should be the good name for the construction of a homotopy Schouten structure given by~\eqref{eq.newh}. Possibly, `coboundary' should be preferred over `triangular' in view of the just said.

Note also that the substitution $p_a+\p_af$ for $p_a$ is a kind of `gauge transformation' (it is an actual gauge transformation in classical mechanics of a charged particle). Since  the shift of the argument formula for  $H'$ above is   a special  canonical transformation , one can imagine a more general canonical transformation applied to $H$ and giving a new `master Hamiltonian'.


\subsection{Graded case. Quasi-triangular $L_{\infty}$-bialgebroids and their odd analog.}

The above obviously works for graded manifolds with appropriate gradings of all involved objects. Of a particular interest for us is the following special setup (which was the starting point of our analysis).

Start from an $L_{\infty}$-algebroid $E$ over $M$ and an   even function $r\in \fun(\Pi E^*)$. Let $Q=Q_E$ be the homological vector field on $\Pi E$ corresponding to the $L_{\infty}$-algebroid structure in $E$. Take the linear Hamiltonian corresponding to it, i.e. the function $H=Q^a(x,\x)p_a+Q^k(x,\x)\pi_k$ on $T^*(\Pi E)$. By applying the Mackenzie--Xu transformation, we obtain a Hamiltonian $H^*\in\fun(T^*(\Pi E^*))$. It defines the $\Sinf$-structure on $\Pi E^*$ (corresponding to the given $L_{\infty}$-algebroid structure in $E$). Perform the shift of the argument of $H^*$ by the derivative of the function $r$ as described in the previous subsection. We shall get a new $\Sinf$-structure. This will introduce into $E$ a structure of a \emph{quasi-triangular $L_{\infty}$-bialgebroid}. One can see that it includes the case of higher Koszul brackets, i.e. the (quasi)triangular $L_{\infty}$-bialgebroid structure induced by an $\Pinf$-structure on $M$, and generalizes the Mackenzie--Xu construction of quasi-triangular Lie bialgebroids and Drinfeld's quasi-triangular Lie bialgebras. 

An odd analog of this construction is based on $\Pi T^*(\Pi E)$ instead of $T^*(\Pi E)$ and the analog of the Mackenzie--Xu transformation introduced in~~\cite{tv:graded}.





\def\cprime{$'$} \def\cprime{$'$}

\end{document}